\documentclass[reqno,11pt]{amsart}

\newtheorem{Thm}{Theorem}
\newtheorem*{Thm*}{Theorem}

\newtheorem{Lemma}{Lemma}

\newtheorem{Corollary}{Corollary}

\def\com#1{\quad\text{#1}\quad}
\def\mc{\mathcal}
\def\tm{t_{-\infty}}
\def\tp{t_{+\infty}}
\def\tpm{t_{\pm\infty}}
\def\Um{U_{-\infty}}
\def\Up{U_{+\infty}}
\def\Upm{U_{\pm\infty}}
\def\NK{{N_{\mc K}}}
\def\ep{\epsilon_\phi}
\usepackage{tikz}
\usepackage{enumerate}

\usepackage{lmodern}

\begin{document}

\title[Inversion of a Non-uniform Difference Operator]%
{Inversion of a Non-uniform\\ Difference Operator}

\author{Blake Temple}
\address{Department of Mathematics,
  University of California, Davis, CA 95616}

\author{Robin Young}
\address{Department of Mathematics and Statistics,
  University of Massachusetts, Amherst, MA 01003}
\date{\today}

\begin{abstract}
The problem of applying Nash-Moser Newton methods to obtain periodic
solutions of the compressible Euler equations has led authors to
identify the main obstacle, namely, how to invert operators which
impose periodicity when they are based on non-uniform shift operators.
Here we begin a theory for finding the inverses of such operators by
proving that a scalar non-uniform difference operator does in fact have a
bounded inverse on its range.  We argue that this is the simplest
example which demonstrates the need to use direct rather than Fourier
methods to analyze inverses of linear operators involving nonuniform
shifts.
\end{abstract}

\maketitle

\section{Introduction}

In this article we explicitly invert the operator which imposes
periodicity for a scalar shift, namely,
\[
\Delta_\Phi v:= \mathcal{S}_{\Phi}v-v=w,
\]
where the shift operator $\mathcal{S}_{\Phi}$ is defined by
\begin{equation}
\label{PeriodicShift2}
\mathcal{S}_{\Phi} v = v\circ\Phi, \com{so that}
\mathcal{S}_{\Phi} v(t) = v(\Phi(t)),
\end{equation}   
and where the non-uniform shift $\Phi(t)$ has the form
\begin{equation}\label{PeriodicShift3}
\Phi(t)=t+\alpha\phi(t).
\end{equation} 
We call $\phi$ the perturbation and $\alpha$ the amplitude, and say
the perturbation is non-degenerate if it is Lipschitz continuous, and
satisfies the further conditions
\[
\phi'(t_*) \ne 0 \com{whenever} \phi(t_*) = 0,
\]
and if at any zero $t_*$ of $\phi$, we have the Taylor estimate
\begin{equation}\label{taylor}
\phi(t) - \phi(t_*) - \phi'(t_*)\,(t-t_*) = O\big((t-t_*)^2\big).
\end{equation}
This is the condition that $\phi$ be better than differentiable, but
not quite twice differentiable, at each zero.  In particular, any
perturbation $\phi$ which is approximately sinusoidal is non-degenerate.

We introduce the problem of inverting $\mathcal{S}_{\Phi}-\mc I$ as a
warmup problem for inverting the linearized operators associated with
the nonlinear operators in \cite{tempyo1,tempyo2,tempyo3,tempyo4},
which impose periodicity for the compressible Euler equations.  In
this case genuine nonlinearity enters as $\phi(t)=\sin(t)+O(\alpha)$
for small $\alpha$.  For motivation, note that the time one evolution
map of the nonlinear Burgers equation
\[
u_t+u\,u_x=0
\]
is a non-uniform shift of the form (\ref{PeriodicShift2}), so a
time-periodic solution would satisfy $\Delta_\Phi v=0$.  Since Burgers
admits no smooth time-periodic solutions, we expect the linear
difference operator $\Delta_\Phi = \mathcal{S}_{\Phi}-\mc I$ to admit
no smooth periodic solutions as well.  Here we address the question as
to whether the inverse $\Delta_\Phi^{-1}$ is bounded.
In \cite{tempyo1,tempyo2}, the authors derive a $2\times2$ linearized
system of the form $\mc S_\Phi - \mc J$, where $\mc S_\Phi$ is a
diagonal shift operator and $\mc J$ is a fixed constant linear
operator; in this case, the shifts have the form
\[
\phi(t) = \sin(t) + O(\alpha), \com{so that}
\Phi(t) = t + \alpha\,\sin(t) + O(\alpha^2).
\]
Inversion of the operator $\mc S_\Phi-\mc J$ is a critical step in the
proof of existence of space- and time-periodic solutions of the
compressible Euler equations by Nash-Moser methods, but the requisite
estimates for the inverses of such operators involving shifts is
beyond current mathematical technology.  The authors propose the
analysis of $\Delta_\Phi$ here as the first step in a program to
develop a mathematical framework for inverting linearized shift
operators in general.  The equation $\Delta_\Phi v=0$ is an example of
an iterative functional equation such as those treated in \cite{kucz}
from a different point of view.

When $\phi(t)=0$, $t$ is a rest point of $\Phi$, and it follows that
the interval between any two consecutive roots of $\phi(t)$ is mapped
to itself under $\Phi$, provided $\alpha$ is small enough.  Denoting
these roots by $\tpm$, it follows that
\begin{equation}
\label{phicond}
\phi(t_{-\infty})=0=\phi(t_{+\infty}), \com{and}
\phi(t)\neq0,\ \ t\in(t_{-\infty},t_{+\infty}),
\end{equation}
and in particular $\Phi:[\tm,\tp]\to[\tm,\tp]$, and
\[
\mathcal{S}_{\Phi}: L^{\infty}{[t_{-\infty},t_{+\infty}]}\to
L^{\infty}{[t_{-\infty},t_{+\infty}]}.
\]

Thus to isolate the essential issue, we address the case when the
domain of $t$ is the interval $[t_{-\infty},t_{+\infty}]$, and $\Phi$
is a monotonic map of this interval to itself.  We introduce a
framework to invert the operator $\mathcal{S}_{\Phi}-\mc I$ on its
range in subspaces of $L^\infty[t_{-\infty},t_{+\infty}]$.  In
particular, we derive estimates for the inverse of $\Delta_{\Phi}$ in
$C^{p}[t_{-\infty},t_{+\infty}]$, the subspace of
$L^{\infty}[t_{-\infty},t_{+\infty}]$ consisting of functions with $p$
continuous derivatives, under the assumption that $\alpha$ is
sufficiently small.

We begin by proving that the operator $\mathcal{S}_{\Phi}-\mc I$
admits a large class of solutions in $L^{\infty}[\tm,\tp]$ in its
kernel, and we characterize them.  However, if we require continuity
at either endpoint $t=t_{\pm\infty}$, we show that there are no
non-constant elements of the kernel.  We then obtain a formula for the
inverse $\Delta_\Phi^{-1}$ by directly solving
$\mathcal{S}_{\Phi}v-v=w$ for $v$ in terms of $w$.  The resulting
formula leads to a solvability condition for $w$, namely,
\begin{equation}
\label{PS4}
\left|\sum_{k\geq0}w(\Phi^kt)\right|<\infty,\quad
\left|\sum_{k\geq0}w(\Phi^{-k}t)\right|<\infty,\quad
\sum_{k}w(\Phi^kt)=Const,
\end{equation} 
for every $t\in[t_{-\infty},t_{+\infty}]$.  We show that condition
\eqref{PS4} characterizes the image of $\mathcal{S}_{\Phi}-\mc I$, and
the image of $\Delta_\Phi$ in the space of Lipshitz continuous
functions consists of precisely those Lipshitz continuous functions
$w$ that meet condition \eqref{PS4}.  Moreover, we prove a bound on
the inverse of $\mathcal{S}_{\Phi}-\mc I$.  We denote the space of
Lipshitz continuous functions by $C^{0,1}[\tm,\tp]$, with norm
\[
\|v\|_{C^{0,1}}=\|v\|_{Lip}+\|v\|_{Sup},
\]
where $\|v\|_{Lip}$ is the minimal Lipshitz constant.  As long as the
perturbation $\phi$ is Lipshitz, the map $\Delta_\Phi=\mc S_\Phi-\mc
I$ is a bounded map from $C^{0,1}[\tm,\tp]$ to itself, and the
following theorem describes its inverse.

\begin{Thm*}
  Assume that $\phi(t)$ satisfies (\ref{taylor}) and (\ref{phicond}).
  Then there exist constants $\alpha_{\phi}$ and $K_{\phi}$, given
  explicitly in terms of $\phi$, such that, if $\alpha<\alpha_{\phi}$
  and $w\in C^{0,1}[\tm,\tp]$ satisfies $w(\tm)=w(\tp)=0$ together
  with the solvability condition \eqref{PS4}, then the equation
  \[
  \Delta_\Phi v(t) = v(\Phi t) - v(t) = w(t)
  \]
  has a solution $v\in C^{0,1}[\tm,\tp]$, uniquely determined up to
  constant, which satisfies
  \[
  \|v\|_{Lip}\leq K_{\phi}\,\|w\|_{Lip},
  \com{so that}
  \|\Delta_\Phi^{-1}\|_{Lip} \le K_\phi.
  \]
\end{Thm*}

The theorem shows that $\Delta_{\Phi}:C^{0,1}\to C^{0,1}$ has a
bounded inverse on its range, which consists of those Lipshitz
functions satisfying the solvability condition.  In Section
\ref{sinvert} we prove that this result extends to invertibility in
the space $C^{p}$, $p\in\mathbb N$, for $\alpha$ sufficiently small,
c.f.~Theorem \ref{Cpnorm} below.  Indeed, we show that there exist
constants $\alpha_0>0$ and $K$ depending only on $\phi$ and $p$, such
that
\begin{equation}
  \label{pest}
  \big\|\Delta_\Phi^{-1}w\big\|_p \le
  \frac{K}{|\alpha|}\,\big\|w\big\|_p,
  \com{for} |\alpha|\le \alpha_0.
\end{equation}

The solvability condition is essential because we obtain two
apparently independent expressions for $v$ in terms of $w$, one of
which converges, while the other need not.  Our formula requires
forward and/or backward iteration of the shift $\Phi$, and it is not
\emph{a priori} clear that both iterations converge.  The solvability
condition is then a compatibility between the two expressions, which
yields a unique answer, and provides convergence in all cases.

Writing $\Delta_{\Phi}=\mathcal{S}_{\Phi}-\mc I$, it is insightful to
compare $\Delta_{\Phi}$, a linear difference operator, to the
derivative operator $D$.  For the derivative operator, we know that
the range, say on $C^p$, lies in the much larger space $C^{p-1}$, so
$D$ can only be inverted on a subspace of $C^{p-1}$ which contains
less regular functions (by one order) than its domain.  In contrast,
the difference operator $\Delta_{\Phi}$ has a range equal to a {\it
  subset} of the space of all Lipschitz functions, which is in fact
the domain of $\Delta_{\Phi}$.  So the difference operator is much
more well behaved than the derivative operator.  Nonetheless, the
standard way of estimating the inverse of $\Delta_{\Phi}$ fails
precisely because $\Delta_{\Phi} v$ does not bound $D v$.  Assume for
example that $\mathcal{S}_{\Phi}$ and $D$ are defined on $C^{p}$.
Here is how a standard argument for estimating the inverse would go:

Let $\mathcal{D}: C^{p}\to C^{p}$ be a general operator we would like
to invert, and assume $v=0$ is the only solution to $\mathcal{D} v=0$;
We'd like to prove
\begin{equation}\label{PS7.0}
\|v\|_{C^{p}}\leq Const.\|\mathcal{D} v\|_{C^{p}}.
\end{equation}   
Assume not, so there exist $v_k$, which we can assume have unit
length, such that
\[
1 = \|v_k\|_{C^{p}}> k\|\mathcal{D} v_k\|_{C^{p}},
\com{so that}
\|\mathcal{D} v_k\|_{C^{p}}\to0.
\]
Now assume (and here is the point at which the argument fails for
$\Delta_{\Phi}$) that $\mathcal{D}$ dominates the derivative in the
sense that
\begin{equation}\label{PS8}
  \|\mathcal{D} v\|_{C^{p}}\geq\|Dv\|_{C^{p}}\equiv \|v\|_{C^{p+1}}.
\end{equation}
Then we would have $\|Dv_k\|_{C^{p}}$ uniformly bounded, and this
together with the bound $\|v_k\|_{C^{p}}=1$ would imply that $v_k$ is
compact in $C^{p}$, and hence a subsequence $v_k\to \bar{v}$ would
converge in $C^{p}$, and the limit would not be zero.  But (\ref{PS8})
implies that $\mathcal{D}\bar{v}=0$, contradicting the assumption that
the $\mathcal{D} v=0$ has only the zero solution.  Thus for this
argument, we need to establish (\ref{PS8}) in order to obtain a bound
on the inverse of $\mathcal{D}$, but for shift operators, the estimate
(\ref{PS8}) clearly fails!  The purpose of this note is to develop
methods sufficient to prove the desired estimate (\ref{PS7.0}) for the
inverse of $\Delta_{\Phi}$ in this case when (\ref{PS8}) fails.

The theorem as stated above gives an estimate for the bound on the
inverse of $\Delta_{\Phi}$ in the Lipschitz norm for shift operators
$\Phi$ that meet the admissibility conditions (\ref{taylor}) and
(\ref{phicond}).  The constant that bounds the inverse is $K_{\phi}$,
defined in (\ref{Kphi}) below, and this depends on seven parameters,
as expressed in Corollary \ref{cormain} below.  Although complicated,
the main point is that the constant $K_{\phi}$ depends continuously on
the norm of $\phi$ for any norm in which these seven parameters are
continuous, such as say $C^2$.  Similarly, the estimate \eqref{pest}
holds if $\phi\in C^{p+1}$, and the constants $\alpha_0$ and $K$
depend on $\|\phi\|_{C^{p+1}}$.

Part of the problem of obtaining bounds on the inverses of non-uniform
difference operators is that the {\it difference} between two shift
operators is an operator that {\it loses a derivative}
\cite{majd,youn}.  As a result, it is not possible to estimate a nearby
linearized operator by an estimate for the unperturbed operator plus
an estimate for the difference in the same norm.  However, our results
here show that estimates for the inverse of $\Delta_{\Phi}$ in $C^p$
are nonetheless stable under perturbation of $\phi$ in $C^{p+1}$
because they depend on constants which depend continuously on the
$C^{p+1}$ norm of $\phi$, (c.f. (\ref{depends}) of Corollary
\ref{cormain} and (\ref{dependsCk}) of Theorem \ref{Cpnorm} below).
Essentially, the above loss of derivative has been transferred to the
perturbation $\phi$, making it consistent with the use of a Nash-Moser
iteration in which the function $\phi$ is mollified at each step of
the iteration \cite{mose,tempyo4}.  The fact that the methods here do not
rely on bounds on the difference between nearby shifts, (which is
worse than the shifts themselves), makes the methods here that much
more interesting, and maybe even a bit surprising.  Our hope is that
the results established here for the scalar non-uniform difference
operator $\mathcal{S}_{\phi}-\mc I$ will shed light on the more
complicated multi-component shift operators whose inverse is required
to complete the authors' program for establishing periodic solutions
of the compressible Euler equations, \cite{tempyo4}.

The solvability condition \eqref{PS4} is non-local and difficult to
verify in general.  However, this problem came about as a result of
a Nash-Moser iteration, which requires solving the linear equation
\[
D\mc F_k(U_k)[V_{k+1}] = \mc F_k(U_k)
\]
for $V_{k+1}$, (the Newton iteration step), where $\mc F_k$ is the
nonlinear operator and $D\mc F_k(U_k)$ is its linearization around
$U_k$.  In our application, $\mc F_k$ includes nonlinear evolution of
a hyperbolic equation, which in turn yields a particular shift
operator.  It is somewhat surprising that the linearization
$DF_k(U_k)$ can be expressed in terms of the same shift operator, so
that the same shift appears on both sides of the equation, which
indicates that the solvability condition may naturally be satisfied in
the Nash-Moser iteration.

Finally, it is important to comment on the substantial difficulties
encountered by trying to invert $\Delta_{\Phi}=\mathcal{S}_{\Phi}-\mc
I$ by Fourier methods.  In fact, we know of no argument employing
Fourier analysis sufficient to obtain and estimate the inverse of
$\Delta_{\Phi}$.  The reason is that the shift operator does not have
a nice expression in terms of standard Fourier modes.  Now the Fourier
framework is most natural for expressing the periodicity conditions in
the base linearized problem which has constant coefficients, but under
perturbation, the Fourier method encounters shift operators like $\mc
S_{\Phi}$.  For example, suppose that $\phi(t)=\sin t$, and consider
$\mc S_\Phi$ acting on $L^2[0,\pi]$ via $\mc S_\Phi v(t)=v(t+\alpha
\sin t)$.  Expanding $v$ in Fourier sine series, we can write
\[
v(t)=\sum_{k=0}^{\infty}a_k\,\sin(kt).
\com{so that} 
\mc S_{\Phi}v(t) = \sum_{k=0}^{\infty}a_k\,e_k(t),
\]
where we have set
\[
e_k(t) = \sin(k(t+\alpha\sin t)).
\]
It is easy to see that $e_k(t)$ is an orthonormal basis for
$L^2[0,\pi]$ with weight function $\sigma=1/(1+\alpha\phi'(t))$, so
that $\mathcal{S}_{\Phi}$ defines a linear isometry on $L^2[0,\pi]$.
Our results here show that there are many fixed points of $\mc S_\Phi$
in $L^2$, but imposing continuity rules out all but the constants, and
we know of no way to establish this via Fourier methods.  The problem
of inverting $\Delta_{\Phi}$, begins with determining the kernel of
$\Delta_{\Phi}$, which reduces to the problem of expressing $e_{k}$ in
terms of the standard Fourier basis $\sin(kt)$.  This leads to the
formula
\[
  e_k(t)=\sum_{j=0}^{\infty}
  \int_{0}^{\pi}\sin(jt)\,\sin(k(t+\alpha\sin t))\;dt,
\]
which is problematic because of the essentially random distribution of
coefficients in the high modes, even for small $k$.  In particular,
our results imply that although the change of basis map has many fixed
points in $L^2$, there are no nontrivial continuous fixed points, so
none in $H^1$.  We know of no way to produce this result directly by
Fourier methods.

\section{Structure of $\Delta_\Phi$}

Recall that we defined the non-uniform difference operator by
\[
\Delta_\Phi = \mc S_\Phi - \mc I, \com{so that}
\Delta_\Phi v(t) = v(\Phi t)-v(t).
\]
For $\alpha$ small enough, the shift function $\Phi$ is invertible.
We write $\Phi^kt$ to denote $k$-fold composition of the shift
function $\Phi$, and we let $t_{-\infty}$ and $t_{+\infty}$ denote
consecutive zeroes of $\phi$, or equivalently fixed points of $\Phi$.
It follows that for any $t$ between $t_{\pm\infty}$, we have
$\Phi^kt\to t_{\pm\infty}$ as $k\to\pm\infty$, and indeed $\Phi$ maps
the interval $[t_{-\infty},t_{+\infty}]$ to itself.  Here we are
implicitly assuming $\alpha\phi>0$; if not we would have
$t_{+\infty}<t_{-\infty}$.

Since $\Phi$ is Lipshitz, it is clear that that $\Delta_\Phi$ can be
regarded as a map $C^{0,1}[t_{-\infty},t_{+\infty}]\to
C^{0,1}[t_{-\infty},t_{+\infty}]$.  Here $C^{0,1}$ is the space of
Lipshitz functions with norm
\[
\|v\|_{C^{0,1}}=\|v\|_{Lip}+\|v\|_{Sup},
\]
where $\|v\|_{Sup}$ denotes the supnorm, and
\begin{equation}\label{Lip}
  \|v\|_{Lip}=\sup_{t_{-\infty}\leq t_a,t_b\leq t_{+\infty}}
  \frac{|v(t_b)-v(t_a)|}{|t_b-t_a|}.
\end{equation}

If $v\in C^{0,1}[t_{-\infty},t_{+\infty}]$, the following
estimate follows immediately:
\[
\|\Delta_{\Phi}v(t)\|_{Lip}\leq \big(1+\|\Phi\|_{Lip}\big)\,\|v\|_{Lip}.
\]
That is, $\Delta_\Phi$ is a bounded linear map from $C^{0,1}$ to
itself, and our goal here is to examine the extent to which it is
invertible.  To do this we first identify the kernel and range of
$\Delta_\Phi$.

\subsection{The Kernel}

We first characterize the kernel of
$\Delta_{\Phi}=\mathcal{S}_{\Phi}-\mc I$ in the space
$L^{\infty}[t_{-\infty},t_{+\infty}]$, and then show that the only
continuous functions in the kernel are the constant functions.

First, choose any $t_0\in(t_{-\infty},t_{+\infty})$, and note that we
can write the interval as a disjoint union,
\[
\big(t_{-\infty},t_{+\infty}\big) =
\bigcup_k\big[\Phi^kt_0,\Phi^{k+1}t_0\big).
\]
It follows that if
\[
v \in \ker\Delta_\Phi, \com{so that}
v(\Phi t) = v(t),
\]
then for any $t$ and any $k$, we have 
\[
v(\Phi^kt) = v(t).
\]

For fixed $t_0$, any $\tilde{v}_0\in L^{\infty}[t_0,\Phi t_0)$ thus
determines an $L^\infty$ element of the kernel, via
\begin{equation}\label{PS9}
  v(t)=\tilde v_0(\Phi^kt), \com{$k$ such that} \Phi^kt\in[t_0,\Phi t_0),
\end{equation}
and every $L^{\infty}$ element of the kernel of $\Delta_{\Phi}$ must
come from some such $\tilde{v}_0$.  Thus (\ref{PS9}) puts
the kernel of $\Delta_{\Phi}$ in $1-1$ correspondence with
$L^{\infty}[t_0,\Phi t_0)$, and the resulting solutions are defined
independent of what base point $t_0$ is chosen.

Because $\Phi^kt\to t_{\pm\infty}$ as $k\to {\pm\infty}$, all the
values of $\tilde{v}_0$ occur in any neighborhood of $t_{\pm\infty}$
for large enough $k$, so that elements of the kernel will become
discontinuous at $t_{\pm\infty}$ unless $\tilde{v}_0(t)=V_0 = Const$
for all $t\in[t_0,\Phi t_0)$.  We state this as a lemma:
  
\begin{Lemma}
Let $t_0\in(t_{-\infty},t_{+\infty})$.  Then the kernel of
$\Delta_{\Phi}$ in $L^{\infty}(t_{-\infty},t_{+\infty})$ is
characterized as the set of all functions $\tilde{v}_0\in
L^{\infty}[t_0,\Phi t_0)$ extended to all of
  $L^{\infty}(t_{-\infty},t_{+\infty})$ by the shift condition
  (\ref{PS9}).  Moreover, the only continuous solutions that lie in
  the $L^{\infty}$ kernel of $\Delta_{\Phi}$ are constants, $v(t)=V_0$.
\label{lem:ker}
\end{Lemma}

Since the space of Lipshitz continuous functions $C^{0,1}$ contains
only continuous functions, we have the following corollary:

\begin{Corollary}
The only elements of the kernel of the difference operator
$\Delta_{\Phi}$ in the space $C^{0,1}[t_{-\infty},t_{+\infty}]$ are
the constant functions $v(t)=V_0$.
\end{Corollary}

\subsection{The Range of $\Delta_{\Phi}$}

We next obtain a condition on $w$ to be in the range of
$\Delta_{\Phi}$ for continuous inputs $v\in
C[t_{-\infty},t_{+\infty}]$.  Thus write
\[
w = \Delta_{\Phi} v=\mathcal{S}_{\Phi}v-v.
\]
We solve directly for $v$ in terms of $w$.   To do so, write
\begin{align*}
v(\Phi t)&=v(t)+w(t)\\
v(\Phi^2 t)&=v(\Phi t)+w(\Phi t)\\
&\vdots\\
v(\Phi^k t)&=v(\Phi^{k-1}t)+w(\Phi^{k-1}t),
\end{align*}
and combine these to obtain
\[
v(t)=v(\Phi^k t)-\sum_{j=0}^{k-1}w(\Phi^jt),
\]
for $k>0$.  Replacing $t$ by $\Phi^{-k}t$, we also get
\[
v(t)=v(\Phi^{-k}t)+\sum_{l=1}^{k}w(\Phi^{-l}t).
\]
Now taking the limit $k\to\infty$ in these gives
\begin{equation}\label{R4}
\begin{aligned}
v(t)&=v(t_{+\infty})-\sum_{j=0}^{\infty}w(\Phi^jt), \com{and}\\
v(t)&=v(t_{-\infty})+\sum_{l=1}^{\infty}w(\Phi^{-l}t),
\end{aligned}
\end{equation}
since $v$ is continuous at $t_{\pm\infty}$.  Equating these yields a
condition on any function $w$ in the range of $\Delta_{\Phi}$ with
continuous input, namely
\begin{equation}\label{R6}
\sum_{k=-\infty}^{+\infty}w(\Phi^{k}t)=Const=v(t_{+\infty})-v(t_{-\infty}).
\end{equation}
Since $w(t)=v(\Phi t)-v(t)$, $w$ is continuous when $v$ is continuous,
so, by construction, we have proven the following lemma which
characterizes the range of $\Delta_{\Phi}:$

\begin{Lemma}\label{Consistency}
The range of the operator $\Delta_{\Phi} v=\mathcal{S}_{\Phi}-\mc I$
on $C[t_{-\infty},\tp]$ is the set of all $w\in
C[t_{-\infty},\tp]$ such that the infinite sums converge for
all $t$,
 \begin{equation}\label{R7aa}  
   \left|\sum_{j=0}^{\infty}w(\Phi^jt)\right|<\infty,\qquad
   \left|\sum_{k=1}^{\infty}w(\Phi^{-k}t)\right|<\infty,
\end{equation}
and such that \eqref{R6} holds, namely
\[
\sum_{k=-\infty}^{+\infty}w(\Phi^{k}t)=Const.
\]
In this case, either equation in (\ref{R4}) defines $v(t)$ such that
$\Delta_{\Phi}v=w$ to within a constant.
\end{Lemma}

\section{Invertibility of $\Delta_{\Phi}$}

Our goal now is to prove that the difference operator $\Delta_\Phi$ is
invertible on its range; that is, given a Lipshitz function $w$
satisfying \eqref{R7aa} and \eqref{R6}, we construct 
a function $v$ satisfying $\Delta_\Phi v=w$, and derive appropriate
bounds on the solution $v$.  We know from \eqref{R4} how the solution
$v$ must be defined, and the main issue is to show that \eqref{R6}
is a sufficient condition for solvability, and then to obtain bounds on $v$ in terms of $w$.  We begin by examining the iterations of the shift function
 in more detail.

\subsection{Properties of the shift}

Recall that the shift function is defined by \eqref{PeriodicShift3},
namely
\[
\Phi(t) = t + \alpha\,\phi(t),
\]
where we assume $\phi$ is non-degenerate, so that
(\ref{PeriodicShift3})-(\ref{taylor}) hold.  For definiteness, assume
that $\alpha\,\phi(t)>0$ on the interval $(\tm,\tp)$, which in turn
implies $\phi'(\tp)<0<\phi'(\tm)$; similar statements and estimates
hold if $\alpha\,\phi(t)<0$ and $\tp<\tm$.  We begin by bounding
$\phi(t)$ away from zero by a trapezoid, and use this to show that for
any $\alpha\ne0$, a finite number of iterations takes us from a
neighborhood of $\tm$ into a neighborhood of $\tp$.

\begin{figure}[thb]
\begin{tikzpicture}[domain=0:6.29]
  \draw[] (-0.4,0) -- (6.7,0);
  \draw[color=red,thick]   plot (\x,{0.7*sin(\x/2 r)}) node[above] {$\alpha\phi(t)$};
  \fill (0,0) circle (2 pt) node[below] {$\tm$};
  \fill (6.283,0) circle (2 pt) node[below] {$\tp$};
  \foreach \x in {1,1.3356,1.7691,2.3106,2.9511,3.6479,4.3256}
    \draw[color=blue,thick] (\x,0) -- (\x,{0.7*sin(\x/2 r)}) --
    ({\x+0.7*sin(\x/2 r)},0);
  \fill[blue] (1,0) circle (1 pt) node[below] {$t_0$};
  \fill[blue] (4.9065,0) circle (1 pt) node[below] {$\Phi^7t_0$};
  \draw[] (-0.4,-2) -- (6.7,-2);
  \draw[color=red,thick]   plot (\x,{0.7*sin(\x/2 r)-2}) node[above] {$\alpha\phi(t)$};
  \fill (0,-2) circle (2 pt) node[below] {$\tm$};
  \fill (6.283,-2) circle (2 pt) node[below] {$\tp$};
  \draw[color=green,thick] (0,-2) -- (1,-1.66440) -- (5.283,-1.66440) -- (6.283,-2);
\end{tikzpicture}
\caption{The Shift Function}
\label{phi}
\end{figure}

From this point on in the paper, assume without loss of generality
that $\phi(t)\geq0$, so that $\phi'(t_{-\infty})>0$,
$\phi'(t_{+\infty})<0$, and that $\alpha>0$, so $\tm<\tp$.  All proofs
carry over to the case $\alpha\phi'(t_{-\infty})>0$,
$\alpha\phi'(t_{+\infty})<0$ with slight modification.

\begin{Lemma}
  \label{lem5}
There are Lipshitz functions $E_+(t)$ and $E_-(t)$, vanishing at $\tp$
and $\tm$, respectively, such that, for both $\pm$ we have,
\begin{equation}\label{Ebound}
\phi(t)=\phi'(t_{\pm\infty})\left\{1+E_{\pm}(t)\right\}(t-t_{\pm\infty}).
\end{equation}
In particular, given any $\delta>0$, there are $\ep>0$
and $m_\phi>0$ such that
\begin{equation}
\begin{aligned}
  |E_-(t)|&<\delta \com{for} \tm \le t \le \tm + \ep,\com{and}\\
  |E_+(t)|&<\delta \com{for} \tp - \ep \le t \le \tp,
\end{aligned}
\label{Epm}
\end{equation}
and
\[
|\phi(t)| \ge m_\phi \com{for}
\tm + \ep\le t \le \tp - \ep.
\]
\end{Lemma}

\begin{proof}
Undoing \eqref{Ebound}, for $t\ne\tpm$, define the functions $E_\pm$
by
\[
\begin{aligned}
E_\pm(t) &:= \frac{\phi(t)}{\phi'(\tpm)\,(t-\tpm)}-1 \\
&= \frac{\phi(t)-\phi(\tpm)-\phi'(\tpm)\,(t-\tpm)}
        {\phi'(\tpm)\,(t-\tpm)}.
\end{aligned}
\]
Since $\phi$ is Lipshitz, $E_\pm$ are Lipshitz away from $\tpm$,
respectively, so we must show boundedness and continuity at $\tpm$.
It follows directly from  \eqref{taylor} and \eqref{phicond} that $E_\pm$ can
be extended to Lipshitz functions on all of $[\tm,\tp]$, and that
$E_\pm(\tpm)=0$.

Now if $\delta>0$ is given, let $E_\phi$ be a Lipshitz bound for
$E_\pm$,
\begin{equation}
  \label{epsilonphi}
  \|E_\pm\|_{Lip} \le E_\phi, \com{and take}
  \ep <\frac{\delta}{E_\phi},
\end{equation}
so that \eqref{Epm} holds.

Finally, we can choose $m_\phi>0$ because $\phi(t)$ is non-zero on the
compact interval $[\tm+\ep,\tp-\ep]$.  Define
\begin{equation}
  \label{phiinf}
\phi'_{\infty} = \min\left\{\big|\phi'(t_{-\infty})\big|,
                          \big|\phi'(t_{+\infty})\big|\right\},
\end{equation}
and restrict the size of $\ep$ if necessary, to obtain
\[
m_\phi = \phi'_\infty\,\ep\,(1-\delta).
\]
Then $\phi(t)$ is bounded below by the trapezoid,
\begin{equation}
  \label{lowerbnd}
\phi(t) \ge \begin{cases}
  \phi'_\infty\,(1-\delta)\,(t-\tm), & \tm\le t\le \tm+\ep,\\
  \phi'_\infty\,(1-\delta)\,\ep      & \tm+\ep\le t\le\tp-\ep,\\
  \phi'_\infty\,(1-\delta)\,(\tp-t), & \tp+\ep\le t\le \tp,
\end{cases}
\end{equation}
as pictured in Figure \ref{phi}.
\end{proof}

Recall that $\Phi$ given by \eqref{PeriodicShift3} is invertible for
all $\alpha$ satisfying
\begin{equation}\label{alphabound}
|\alpha| < \alpha_{\phi} := \frac{1}{\|\phi\|_{Lip}}.
\end{equation}
For convenience, assume now, without loss of generality, that
\[
\ep < |\tp-\tm|/4.
\]

\begin{Corollary}
  \label{cor6}
For any non-degenerate $\phi$ and $\alpha$ satisfying
\eqref{alphabound}, there is a positive integer $N_\alpha$ depending
only on $\alpha$ and $m_\phi$, such that for any $k\ge N_\alpha$ and
$t\in[\tm+\ep,\tp-\ep]$, we have both
\[
\Phi^{-k} t\in(t_{-\infty},t_{-\infty+\epsilon_{\phi}}) \com{and}
\Phi^{k} t\in(t_{+\infty}-\epsilon_{\phi},t_{+\infty}).
\]
\end{Corollary}

\begin{proof}
For definiteness, suppose that $\alpha\phi(t)>0$.  As long as $t$ and
$\Phi^jt$ remain in the interval
$[\tm+\ep,\tp-\ep]$, we have
\[
\Phi t \ge t + \alpha\,m_\phi, \com{and}
\Phi^j t \ge t + j\,\alpha\,m_\phi,
\]
so it suffices to take
\begin{equation}
  \label{Nalpha}
N_\alpha \ge \frac{\tp -\ep -
  (\tm + \ep)}{\alpha\,m_\phi}+1.
\end{equation}
Note that if $\alpha\phi<0$, the intervals change: in that case
\[
\Phi^{-k} t\in(t_{-\infty}-\ep,t_{-\infty}) \com{and}
\Phi^{k} t\in(t_{+\infty},t_{+\infty}+\epsilon_{\phi}),
\]
for $t\in[\tp+\ep,\tm-\ep]$.
\end{proof}

We now show that because $\phi$ is non-degenerate, iteration of the
shift map $\Phi$ gives a geometric progression into the fixed points,
so we can use a geometric series to control the cumulative iterations.

\begin{Thm}\label{Phideep}
Let $\delta>0$ be given and let $\ep$, $m_\phi$ and $N_\alpha$ be
determined so that Lemma \ref{lem5} and Corollary \ref{cor6} hold, and
assume $t_a\le t_b$.  If $t_a\ge t_{-\infty}+\epsilon_{\phi}$, then
\begin{equation}\label{Phisumplus}
\sum_{k=0}^{\infty}\left|\Phi^kt_b-\Phi^kt_a\right|\leq
V_{\phi}|t_b-t_a|,
\end{equation} 
and if $t_b\le t_{+\infty}-\epsilon_{\phi}$, then
\begin{equation}\label{Phisumminus}
\sum_{k=0}^{-\infty}\left|\Phi^kt_b-\Phi^kt_a\right|\leq V_{\phi}|t_b-t_a|,
\end{equation} 
where
\begin{equation}\label{VPhisum}
V_\phi = \frac{2\,(1+|\alpha|\,\|\phi\|_{Lip})^{N_\alpha}}
          {|\alpha|\,\phi'_\infty\,(1-2\delta)}.
\end{equation}
\end{Thm}

\begin{proof}
  For definiteness, assume that $t_a \ge \tm + \ep$. We prove
  \eqref{Phisumplus}; a similar argument yields \eqref{Phisumminus}
  when $t_b\le \tp-\ep$.  We break the proof into two cases:
\begin{enumerate}[(i)]
  \item $t_a\ge \tp-\ep$; and 
  \item $t_a \in (t_{-\infty}+\epsilon_{\phi},t_{+\infty}-\epsilon_{\phi})$.
\end{enumerate}

Consider first case (i).  Then also $t_b\ge\tp-\ep$, and we can use
(\ref{Ebound}) to write
\begin{align*}
\phi(t_b)&=\phi'(t_{+\infty})\left\{1+E_{+}(t_b)\right\}(t_b-t_{+\infty}),\\
\phi(t_a)&=\phi'(t_{+\infty})\left\{1+E_{+}(t_a)\right\}(t_a-t_{+\infty}),
\end{align*}
and recall we have assumed $\phi>0$ so that $\phi'(t_{+\infty})<0$.
Subtracting and rearranging, we obtain
\begin{align*}
\phi(t_b&)-\phi(t_a)\\&=\phi'(t_{+\infty})\left\{t_b-t_a+E_{+}(t_b)(t_b-t_{+\infty})-E_{+}(t_a)(t_a-t_{+\infty})\right\}\\
&=\phi'(t_{+\infty})\left\{t_b-t_a+E_{+}(t_b)(t_b-t_a)+(E_{+}(t_b)-E_{+}(t_a))(t_a-t_{+\infty})\right\}\\
&=\phi'(t_{+\infty})\left\{1+E_{+}(t_b)+\frac{E_{+}(t_b)-E_{+}(t_a)}{t_b-t_a}(t_a-t_{+\infty})\right\}(t_b-t_a).
\end{align*}
Now, according to \eqref{Epm}, we have $|E_+(t_b)|\leq\delta$, and by
\eqref{epsilonphi}, we have
\[
\left|\frac{E_{+}(t_b)-E_{+}(t_a)}{t_b-t_a}(t_b-t_{+\infty})\right|\leq E_{\phi}\epsilon_{\phi}\leq\delta.
\]
Thus, since $\phi'(t_{+\infty})<0$, we can estimate
\[
\phi'(t_{+\infty})\left\{1+2\delta)\right\}(t_b-t_a)\leq\phi(t_b)-\phi(t_a)\leq\phi'(t_{+\infty})\left\{1-2\delta)\right\}(t_b-t_a),
\]
which in turn gives
\begin{equation}
  \label{PhiBd}
\big|\Phi t_b-\Phi t_a\big|\leq \left\{1-|\alpha\,\phi'(t_{+\infty})|(1-2\delta)\right\}(t_b-t_a).
\end{equation}
By induction conclude that
\[
\left|\Phi^k t_b-\Phi^k t_a\right|\leq \left\{1-|\alpha\,\phi'(t_{+\infty})|(1-2\delta)\right\}^k(t_b-t_a),
\]
and thus
\begin{align*}
\sum_{k=0}^{\infty}\left|\Phi^k t_b-\Phi^k t_a\right|
&\leq \sum_{k=0}^{\infty}\left\{1-|\alpha\,\phi'(t_{+\infty})|(1-2\delta)\right\}^k(t_b-t_a)\\
&\le V_1\,(t_b-t_a),
\end{align*}
where we have set
\[
V_1 := \frac{1}{|\alpha\,\phi'_{\infty}|\,(1-2\delta)};
\]
this completes the proof in case (i).

Consider now case (ii).  Since $t_{-\infty}+\epsilon_{\phi}\leq
t_a\leq t_b$, by Corollary \ref{cor6} we know that
$t_{+\infty}-\epsilon_{\phi}\leq\Phi^kt_a<\Phi^kt_b$ for all $k\geq
N_\alpha$, and we write
\begin{equation}\label{breaksum1}
  \sum_{k=0}^{\infty}\left|\Phi^kt_b-\Phi^kt_a\right|
  =\sum_{k=0}^{N_\alpha-1}\left|\Phi^kt_b-\Phi^kt_a\right|
  +\sum_{k=N_\alpha}^{\infty}\left|\Phi^kt_b-\Phi^kt_a\right|.
\end{equation} 
Since all the terms in the second sum lie within $\epsilon_{\phi}$ of
$t_{+\infty}$, we can estimate the second sum by case (i) and obtain
\[
  \sum_{k=N_\alpha}^{\infty}\left|\Phi^kt_b-\Phi^kt_a\right|
  \leq V_1\left|\Phi^{N_\alpha}t_b-\Phi^{N_\alpha}t_a\right|.
\]
Now $|\Phi t_2-\Phi t_1|\le (1 +
|\alpha|\,\|\phi\|_{Lip})\,|t_2-t_1|$, so by induction,
\begin{equation}\label{sumPhik}
|\Phi^k t_2-\Phi^k t_1|\le (1 +
|\alpha|\,\|\phi\|_{Lip})^k\,|t_2-t_1|.
\end{equation}

Thus, the second sum in (\ref{breaksum1}) is bounded by
\[
  \sum_{k=N_\alpha}^{\infty}\left|\Phi^kt_b-\Phi^kt_a\right|
  \leq V_1(1+|\alpha|\,\|\phi\|_{Lip})^{N_\alpha}\left|t_b-t_a\right|.
\]

Consider next the first sum in (\ref{breaksum1}). By (\ref{sumPhik})
we can estimate
\[
\begin{aligned}
  \sum_{k=0}^{N_\alpha-1}\left|\Phi^{k}t_b-\Phi^{k}t_a\right|
  &\leq\sum_{k=0}^{N_\alpha-1}(1+|\alpha|\,\|\phi\|_{Lip})^k(t_b-t_a)\\
  &\leq\frac{(1+|\alpha|\,\|\phi\|_{Lip})^{N_\alpha}-1}
      {|\alpha|\,\|\phi\|_{Lip}}\,|t_b-t_a|,
\end{aligned}
\]
where we have summed the finite geometric series.
 
Combining the estimates for the sums in (\ref{breaksum1}), we get
\[
\sum_{k=0}^{\infty}\left|\Phi^kt_b-\Phi^kt_a\right|
\leq V_\phi\,\left|t_b-t_a\right|,
\]
where we have written
\begin{align*}
V_1\,(1&+|\alpha|\,\|\phi\|_{Lip})^{N_\alpha}+
\frac{(1+|\alpha|\,\|\phi\|_{Lip})^{N_\alpha}-1}
     {|\alpha|\,\|\phi\|_{Lip}}\\
&\le \frac{2\,(1+|\alpha|\,\|\phi\|_{Lip})^{N_\alpha}}
          {|\alpha|\,\phi'_\infty\,(1-2\delta)} =: V_\phi,
\end{align*}
since $\phi'_\infty\le \|\phi\|_{Lip}$.  This establishes case (ii),
and completes the proof of the theorem.
\end{proof}

\subsection{Convergence of the infinite sums}

Our goal is to prove invertibility of $\Delta_\Phi$, which amounts to
proving consistency of and estimates for our construction \eqref{R4},
which expresses the solution $v$ of the equation
\[
\Delta_\Phi v = \mc S_\Phi v - v = w
\]
as an infinite series.  We begin by using the geometric convergence of
iterations of $\Phi$ to simplify the condition for convergence of
these series.

\begin{Lemma}\label{lem8}
  Assume that the conditions of Theorem \ref{Phideep} holds, and let
  $w$ be Lipschitz continuous on $[t_{-\infty},t_{+\infty}]$.  Then
  the series in (\ref{R7aa}) converge if and only if $w$ vanishes at
  the fixed points,
\[
 w(\tm)=0 \com{and}  w(\tp) = 0.
\]
Moreover, if $t$ lies within $\epsilon_{\phi}$ of one of the endpoints
$t_{\pm\infty}$, we have the bound
\begin{equation}\label{nicend3}
  \left|\sum_{k=0}^{\infty}w(\Phi^{\pm k}t)\right|
  \leq\|w\|_{Lip}\,\frac{|t_{\pm \infty}-t|}
           {|\alpha|\,(1-\delta)\,\phi'_{\infty}},
\end{equation}
where $\phi'_\infty$ is given by \eqref{phiinf}.
\end{Lemma}

\begin{proof}
Assuming convergence of the series (\ref{R7aa}), 
\[
\left|\sum_{k=0}^{\infty}w(\Phi^kt)\right|<\infty \com{and}
\left|\sum_{k=0}^{\infty}w(\Phi^{-k}t)\right|<\infty,
\]
it follows that $w(\Phi^{\pm k}t)\to 0$,
and so by continuity,
\[
w(\tpm) = w\big(\lim_{k\to\infty}\Phi^{\pm k}t\big) =
\lim_{k\to\infty}w(\Phi^{\pm k}t) = 0.
\]

For the reverse implication, assume that $w(\tp)=0$.  We show the
first sum in (\ref{R7aa}) is finite; the other case is similar.  Since
only finitely many terms in each sum of (\ref{R7aa}) lie
$\epsilon_{\phi}$ away from $t_{\pm\infty}$, it suffices to prove
\[
\left|\sum_{k=0}^{\infty}w(\Phi^kt)\right|<\infty \com{for}
t\in(t_{+\infty}-\epsilon_{\phi},t_{+\infty}).
\]
We write $t_0:=t>\tp-\ep$ and set $t_k = \Phi^kt$.  Using
\eqref{PeriodicShift3} and \eqref{Ebound}, we write
\[
t_{k+1} = t_k + \alpha\,\phi'(\tp)\,\{1+E_+(t_k)\}\,(t_k-\tp),
\]
which in turn yields
\[
\begin{aligned}
  \tp-t_{k+1} &= \big(1 + \alpha\,\phi'(\tp)\,\{1+E_+(t_k)\}\big)
  \,(\tp-t_k) \\
  &\le \big( 1 - |\alpha\,\phi'(\tp)|\,(1-\delta)\big)\,(\tp-t_k),
\end{aligned}
\]
where we have used \eqref{Epm}, and recalled that
$\alpha\,\phi'(\tp)<0$.  Continuing by induction we obtain
\[
t_{\infty}-t_k\leq
\big( 1 - |\alpha\,\phi'(\tp)|\,(1-\delta)\big)^k\,(\tp-t),
\]
and since $w(\tp)=0$, we can write
\begin{align*}
  \left|\sum_{k=0}^{+\infty}w(\Phi^kt)\right|
  &\leq \sum_{k=0}^{\infty}\left|w(\tp)-w(\Phi^kt)\right|
  \leq\sum_{k=0}^{+\infty}\|w\|_{Lip}\,|\tp-t_k|\\
  &\leq\|w\|_{Lip}\,\sum_{k=0}^{+\infty}
\big( 1 - |\alpha\,\phi'(\tp)|\,(1-\delta)\big)^k\,(t_{+\infty}-t)\\
&\leq \|w\|_{Lip}\,\frac{t_{+\infty}-t}
    {|\alpha\,\phi'(\tp)|\,(1-\delta)}<\infty.
\end{align*}
The argument in the symmetrical case
$t\in(t_{-\infty},t_{-\infty}+\epsilon_{\phi})$ gives the result 
\[
\left|\sum_{k=0}^{-\infty}w(\Phi^{-k}t)\right|\leq
\|w\|_{Lip}\frac{t-\tm}{|\alpha\,\phi'(\tm)|\,(1-\delta)},
\]
and these together yield \eqref{nicend3}.
\end{proof}

As a corollary, we obtain bounds on the full sum
$\sum_{k=-\infty}^{\infty}w(\Phi^kt)$ provided $w$ vanishes at $\tpm$.

\begin{Corollary}  Assume $w$ is Lipschitz continuous and satisfies
  $w(\tpm)=0$, and the conditions of Theorem \ref{Phideep} hold.  Then
  for any $t\in[\tm,\tp]$,
\begin{equation}\label{estcombat}
\left|\sum_{k=-\infty}^{\infty}w(\Phi^kt)\right|
\leq \|w\|_{Lip}\,
\left\{\frac{2}{|\alpha|\,(1-\delta)\,\phi'_{\infty}}+N_\alpha\right\}
\,|t_{+\infty}-t_{-\infty}|.
\end{equation}
\end{Corollary}

\begin{proof}
Since the $\Phi^kt\to\tpm$ as $k\to\pm\infty$, we can replace $t$ by
$\Phi^jt$ for any convenient integer $j$.  Using Corollary \ref{cor6},
we choose $t_0 = \Phi^jt$ such that
\[
t_0 \le \tm+\ep \com{and} \Phi^{N_\alpha+1}t_0 \ge \tp-\ep.
\]
Then we have
\[
\left|\sum_{k=-\infty}^{\infty}w(\Phi^kt)\right|\leq
\left|\sum_{k=0}^{-\infty}w(\Phi^{-k}t_0)\right|+
\left|\sum_{k=1}^{N_\alpha}w(\Phi^kt_0)\right|+
\left|\sum_{k=N_\alpha+1}^{\infty}w(\Phi^kt_0)\right|,
\]
and the first and third sums are estimated by \eqref{nicend3}.  We
estimate the middle term by
\[
\left|\sum_{k=1}^{N_\alpha}w(\Phi^kt_0)\right| \le 
N_\alpha\,\|w\|_{Sup}\leq
N_\alpha\,\|w\|_{Lip}\,|t_{+\infty}-t_{-\infty}|,
\]
since $w$ vanishes at $\tpm$.  Combining the three terms yields
\eqref{estcombat}.
\end{proof}

\subsection{Invertibility in the space $C^{0,1}$}

We now state and prove our main theorem, which states that the
operator $\Delta_\Phi = \mc S_\Phi - \mc I$ is invertible on
its range, and that the inverse is bounded in the $C^{0,1}$ norm.

\begin{Thm}\label{Firstinverse}  
  Assume that $\phi(t)$ is non-degenerate, let $\delta>0$ and
  \[
  |\alpha| < \frac{1-\delta}{\phi'_{\infty}}
  < \frac1{\|\phi\|_{Lip}}
  \]
  be given, and choose $\ep < (\tp-\tm)/4$ and $N_\alpha$ so that the
  conditions of Theorem \ref{Phideep} hold.  Also let
  $w\in C^{0,1}[\tm,\tp]$ satisfy $w(\tm)=w(\tp)=0$ and the
  consistency condition \eqref{R6}, namely
  \[
  \sum_{k=-\infty}^{+\infty} w(\Phi^kt) = Const.
  \]
  Then the equation
  \[
  \Delta_\Phi v(t) = v(\Phi t) - v(t) = w(t)
  \]
  has a solution $v\in C^{0,1}[\tm,\tp]$, uniquely determined up to
  constant, which is explicitly given by either equation in
  \eqref{R4}, which are consistent.  Moreover, the solution satisfies
  the bound
  \begin{equation}
  \|v\|_{Lip}\leq K_{\phi}\,\|w\|_{Lip}, \com{so that}
  \|\Delta_\Phi^{-1}\|_{} \le K_\phi,
  \label{InvBound}
  \end{equation}
  where $K_{\phi}$ is given by
  \begin{equation}\label{Kphi}
  K_{\phi}=\max\left\{K_1,V_{\phi}\right\},
  \end{equation}
  where
\begin{equation}\label{K1}
K_1 = \frac{2 + 2\,\ep + |\alpha|\,N_\alpha\,(1-\delta)\,\phi'_{\infty}}
     {|\alpha|\,(1-\delta)\,\phi'_{\infty}\,(\tp-\tm-2\ep)},
\end{equation}
and $V_{\phi}$ is given in (\ref{VPhisum}).
\end{Thm}

\begin{proof}
According to Lemma \ref{lem8}, the conditions $w(\tpm)=0$ imply that
the infinite series converge, so that equations \eqref{R4} make
sense.  We can choose one of $v(\tpm)$ arbitrarily, and the other is
determined by \eqref{R6}, which also implies consistency of both
equations in \eqref{R4}.  Since the solution is given by an explicit
formula, it is unique up to our choice of constant, and all that
remains is to establish \eqref{InvBound}.

As above, for definiteness we assume that $\alpha\,\phi(t)>0$; similar
estimates hold for the other case.  Assume that we are given $t_a$ and
$t_b\in[\tm,\tp]$, and suppose $t_a<t_b$.  We again consider two
cases:
\begin{enumerate}[(i)]
  \item $t_a\ge \tm+\ep$ or $t_b\le \tp-\ep$; and 
  \item $t_a < \tm + \ep$ and $t_b > \tp-\ep$.
\end{enumerate}

For the first case, we assume $t_a\ge \tm+\ep$; the case $t_b\le
\tp-\ep$ follows similarly.  Here Theorem \ref{Phideep} applies
directly: we use the first equation in \eqref{R4} to describe $v(t)$
and write
\begin{align}
  \big|v(t_b)-v(t_a)\big|&=
  \left|\sum_{k=0}^{\infty}\left(w(\Phi^{k}t_b)-w(\Phi^{k}t_a)\right)\right|
  \nonumber\\ &\leq
  \|w\|_{Lip}\sum_{k=0}^{\infty}\left|\Phi^{k}t_b-\Phi^{k}t_a\right|
  \nonumber\\[2pt]
  &\leq\|w\|_{Lip}\,V_{\phi}\,\big|t_b-t_a\big|,
  \label{case1}
\end{align}
where we have applied (\ref{Phisumplus}).  When $t_b\leq
t_{+\infty}-\epsilon_{\phi}$, the same estimate holds using the second
equation of (\ref{R4}) and (\ref{Phisumminus}).  This is the required
estimate for case (i).

We now consider case (ii), in which $t_a$ and $t_b$ lie within
$\epsilon_{\phi}$ of different fixed points.  We use
different expressions for $v(t_a)$ and $v(t_b)$; by \eqref{R4}, we
have
\begin{align*}
  v(t_a) &= v(t_{-\infty})+\sum_{k=1}^{\infty}w(\Phi^{-k}t_a),
  \com{and}\\
  v(t_b) &= v(t_{+\infty})-\sum_{k=0}^{\infty}w(\Phi^kt_b).
\end{align*}
Subtracting, taking the absolute value, and applying (\ref{nicend3})
gives
\begin{equation}
\big|v(t_b)-v(t_b)\big|\leq\big|v(t_{+\infty})-v(t_{-\infty})\big| +
\|w\|_{Lip}\,\frac{2\,\epsilon_{\phi}}{|\alpha|\,(1-\delta)\,\phi'_{\infty}}.
\label{vdiff}
\end{equation}
We use \eqref{estcombat} to write
\[
\big|v(\tp)-v(\tm)\big| \le \|w\|_{Lip}\,
\left\{\frac{2}{|\alpha|\,(1-\delta)\,\phi'_{\infty}}+N_\alpha\right\}
\,|t_{+\infty}-t_{-\infty}|,
\]
so that \eqref{vdiff} becomes
\[
\big|v(t_b)-v(t_b)\big| \le \|w\|_{Lip}\,K_0,
\]
where we have set
\begin{align*}
K_0 &:= \left\{\frac{2}{|\alpha|\,(1-\delta)\,\phi'_{\infty}}+N_\alpha\right\}
\,|t_{+\infty}-t_{-\infty}| +
\frac{2\,\ep}{|\alpha|\,(1-\delta)\,\phi'_{\infty}}\\
&= \frac{2 + 2\,\ep + |\alpha|\,N_\alpha\,(1-\delta)\,\phi'_{\infty}}
     {|\alpha|\,(1-\delta)\,\phi'_{\infty}}.
\end{align*}
Finally, since
\[
t_b - t_a > \tp - \tm - 2\ep,
\]
we can write
\begin{equation}
  \label{case2}
\big|v(t_b)-v(t_a)\big| \le \|w\|_{Lip}\,K_1\,\big|t_b-t_a\big|,
\end{equation}
where we have set
\[
K_1 := \frac{K_0}{\tp-\tm-2\ep} =
\frac{2 + 2\,\ep + |\alpha|\,N_\alpha\,(1-\delta)\,\phi'_{\infty}}
     {|\alpha|\,(1-\delta)\,\phi'_{\infty}\,(\tp-\tm-2\ep)},
\]
which is \eqref{K1}.  Combining \eqref{case1} and \eqref{case2}
yields \eqref{InvBound}, and the proof is complete.
\end{proof}

\begin{Corollary}\label{cormain}
Since the dependence of $K_{\phi}$ in (\ref{K1}) is given by
\begin{equation}
  \label{depends}
K_{\phi}\equiv K_{\phi}(\alpha,\delta,\epsilon_{\phi},N_\alpha,
                       t_{+\infty}-t_{-\infty},\|\phi\|_{Lip},\phi'_{\infty}),
\end{equation}
it follows that $K_{\phi}$ depends continuously on $\phi$ in any norm
in which these parameters are continuous.  In particular, if $\phi$ is
the perturbation of a fixed $\phi_0$, we can write
\[
K_\phi = \frac1{|\alpha|}\,K(\delta,\phi) \com{where} K(\delta,\phi) = O(1).
\]
\end{Corollary}

Indeed, \eqref{Nalpha} yields $N_\alpha= O(1/|\alpha|)$, so that
\eqref{K1} yields $K_1=O(1/|\alpha|)$ and using
\[
(1 + |\alpha|\,\|\phi\|_{Lip})^{N_\alpha}
\le e^{|\alpha|\,\|\phi\|_{Lip}\,N_\alpha} = O(1),
\]
we see that \eqref{VPhisum} also gives $V_\phi=O(1/|\alpha|)$.

\section{Inversion of $\Delta_{\phi}$ in other norms.}\label{sinvert}

We now address the question: in which other norms is $\Delta_\Phi$
invertible?  We consider only functions which are at least Lipshitz
continuous, so that the results of the previous section apply.
Our goal then is to identify those norms in which $\Delta_\Phi$ is
invertible on its range,
\[
\|v\| \le K\,\|w\| \com{if} \Delta_\Phi v = w.
\]

To be specific, let $\|\cdot\|$ denote the norm and let $X$ be the
Banach space of Lipshitz functions on $[\tm,\tp]$ bounded in this
norm; that is,
\[
X = \Big\{ v\in C^{0,1}[\tm,\tp],\  \|v\|<\infty \Big\}.
\]
Our starting assumption is that the norm $\|\cdot\|$ respects
composition, in the sense that
\begin{equation}
  \label{comp}
\|v\circ \Psi\| \le K_0\,\|v\|\,\|\Psi\| \com{for some $K_0$,}
\end{equation}
provided $\Psi:[\tm,\tp]\to[\tm,\tp]$ and $\Psi\in X$.  In particular
the shift operator $\mc S_\Phi$ is a bounded operator on $X$, whose
norm is controlled by $\|\Phi\|$, and so $\Delta_\Phi:X\to X$ is a
bounded operator,
\[
\|\Delta_\Phi\| \le K_0\,\|\Phi\|+1.
\]

The following theorem gives a sufficient condition for the inverse of
$\Delta_\Phi$ to be bounded in a given norm.

\begin{Thm} Assume the Banach space $X$ has a norm satisfying \eqref{comp}.
  The range of the operator $\Delta_\Phi:X\to X$ is the set $\mc R$
  that consists of those functions $w\in X$ that satisfy \eqref{R6},
  which implicitly implies $w(\tpm)=0$.  Moreover, if we can sum either
  of the series
  \begin{equation}
  \sum_{k=0}^\infty\big\|\Phi^{ k}\big\| < \infty, \com{or}
  \sum_{k=1}^\infty\big\|\Phi^{-k}\big\| < \infty,
  \label{normcond}
  \end{equation}
  where as usual, $\Phi^k$ denotes $k$-fold composition, then the
  solution operator $\Delta_\Phi^{-1}$ is bounded on its domain.  That
  is, there is a constant $K$ such that, if $w\in \mc R$ and $v$
  satisfies
  \[
  \Delta_\Phi v = w, \com{then}
  \|v\| \le K\,\|w\|.
  \]
  \label{genthm}
\end{Thm}

We regard \eqref{normcond} as a condition on $\Phi$ (or
$\alpha\,\phi$) which implies that the equation can be solved in $X$:
that is, if the norm satisfies \eqref{comp}, then condition
\eqref{normcond} is a condition only on the shift function which
implies that the equation $\Delta_\Phi$ is invertible on its range in
the space $X$.

\begin{proof}
Since the norm satisfies \eqref{comp}, the forward operator is bounded
and maps $X\to X$.  Since $X$ is a subset of the Lipshitz continuous
functions, the range $\mc R$ consists of those elements of $X$ which
satisfy our consistency conditions $w(\tpm)=0$ and \eqref{R6}.

We must show that if $w \in\mc R$, and if $v$ solves $\Delta_\Phi v =
w$, then $\|v\| \le K\,\|w\|$ for some $K$.  Since the solution $v$ is
given by a formula, namely \eqref{R4}, this bound becomes an explicit
inequality.  To prove the theorem, use the formula of \eqref{R4}
corresponding to the finite sum in \eqref{normcond}, and set the
corresponding $v(\tp)$ or $v(\tm)$ to zero.  For definiteness, suppose
that the first inequality in \eqref{normcond} holds, say
\[
\sum_{k=0}^\infty\big\|\Phi^k\big\| = M;
\]
then set $v(\tp)=0$ and use the first equation in \eqref{R4} to define
$v$.  Using \eqref{comp}, it follows that
\begin{align*}
  \big\|v\big\| &= \left\|\sum_{k=0}^\infty w\circ\Phi^k\right\| \\
  &\le \sum_{k=0}^\infty K_0\,\big\|w\big\|\,\left\|\Phi^k\right\| \\
  &\le K\,\big\|w\big\|,
\end{align*}
with $K = K_0\,M$, as required; the other case follows similarly.
\end{proof}
  
In light of \eqref{normcond}, the following corollary is immediate.

\begin{Corollary}
  If $K_0=1$ in \eqref{comp}, then a sufficient condition for the
  operator $\Delta_\Phi$ to be bounded is that for
  some $\eta>0$ and $k$ large enough, we have
  \[
  \|\Phi^{k+1}\|\le (1-\eta)\,\|\Phi^k\|, \com{or}
  \|\Phi^{-k-1}\|\le (1-\eta)\,\|\Phi^{-k}\|.
  \]
\end{Corollary}

Note that these bounds apply for the Lipshitz norm,
cf.~Theorem~\ref{Phideep} and \eqref{PhiBd}.

\subsection{Bounds in the $C^p$ Norm}

Because the kernel of $\Delta_\Phi$ is infinite dimensional in spaces
of discontinuous functions, it is most natural to establish
invertibility of in spaces in which the function and its
derivatives are continuous.  Here we prove invertibility of
$\Delta_\Phi$ in $C^p$, the simplest such spaces.

Recall that for integers $p\ge 0$, we can write the $C^p$ norm as
\[
\|f\|_p = \sup_t\big|f(t)\big|_p,
\]
where we have set
\[
\big|f(t)\big|_p := \Big|\big(f(t),f'(t),\dots,f^{(p)}(t)\big)^T\Big|
 =\sup_t\sum_{j=0}^p |f^{(j)}(t)|,
\]
that is, $\big|f(t)\big|_p$ is a norm of the vector of derivatives up
to $p$-th order.  It is convenient to work with pointwise values of
the iterates $\Phi^kt$.  The point here is that the order $p$ is
fixed, while the iteration number $k$ is large, so we wish to give an
asymptotic description of $\frac{d^p}{dt^p}\Phi^kt$.  We know that
$\Phi^kt\to\tp$ geometrically for large $k$, and we wish to leverage
this to get bounds for $\|\Phi^k\|_p$.

\begin{Lemma}
For $p\ge 1$, the space $C^p[\tm,\tp]$, endowed with the norm $\|f\|_p =
\sup_t\big|f(t)\big|_p$, is a Banach space containing the
Lipshitz functions and for which the bound \eqref{comp} holds.
\end{Lemma}

Rather than treat the case of general $p$, we treat the case $p=3$,
which illustrates the development for the case of general $p$, which
in turn follows by induction.  We show \eqref{normcond} holds for the
norm $\|f\|_p$, as long as $\phi$ is itself $C^{p+1}$.  The goal is to
take advantage of the geometric convergence of $\Phi^kt$ to the limit
$\tpm$ as $k\to\pm\infty$.

To control the derivatives of $\Phi$, we introduce the
following notation: for a given $t=t_0$, define
\begin{equation}
  \label{notation}
t_k = \Phi^kt, \quad
s_k = (\Phi^kt)', \quad
r_k = (\Phi^kt)'', \com{and}
q_k = (\Phi^kt)'''.
\end{equation}
We now describe the discrete dynamical system for the vector
\[
U_k = (t_k,\;s_k,\;r_k,\;q_k)^T.
\]
To begin, assume $z=z(t)$ is some given
function, and compute
\[
\begin{aligned}
  (\Phi z)' &= \Phi'(z)\,z' = \big(1 + \alpha\,\phi'(z)\big)\,z',\\
  (\Phi z)''&= \big(1 + \alpha\,\phi'(z)\big)\,z''
        + \alpha\,\phi''(z)\,(z')^2, \com{and}\\
  (\Phi z)'''&= \big(1 + \alpha\,\phi'(z)\big)\,z'''
        + 3\,\alpha\,\phi''(z)\,z'\,z''
        + \alpha\,\phi'''(z)\,(z')^3.
\end{aligned}
\]
Now setting $z=\Phi^kt$ and using \eqref{notation}, we obtain the
system
\begin{equation}
  \label{DS}
  \begin{aligned}
  t_{k+1} &= t_k + \alpha\,\phi(t_k), \\
  s_{k+1} &= \big(1 + \alpha\,\phi'(t_k)\big)\,s_k,  \\
  r_{k+1} &= \big(1 + \alpha\,\phi'(t_k)\big)\,r_k
         + \alpha\,\phi''(t_k)\,s_k^2,  \\
  q_{k+1} &= \big(1 + \alpha\,\phi'(t_k)\big)\,q_k
         + 3\,\alpha\,\phi''(t_k)\,s_k\,r_k
         + \alpha\,\phi'''(t_k)\,s_k^3,
  \end{aligned}
\end{equation}
which we regard as a discrete dynamical system,
\begin{equation}
  \label{DDS}
U_{k+1} = F(U_k),
\end{equation}
where $F$ is the RHS of \eqref{DS}.  Note that this is a hierarchy, in
that each successive equation beyond the nonlinear equation for $t_k$
is a linear inhomogeneous equation, once the previous components are
given.  It is clear that the system can be extended to arbitrary fixed
values of $p$.  Also, since each equation beyond the first is linear,
solutions exist and remain bounded for all finite $k$.

It is easy to find the fixed points of \eqref{DS} by directly finding
the rest points of each equation in turn: the first equation yields
$\phi(t)=0$, so $t=\tpm$, and subsequent equations have trivial
solutions because $\phi'(\tpm)\ne 0$, and the fixed points are
\[
U_{-\infty} := (\tm,\;0,\;0,\;0)^T \com{and}
U_{+\infty} := (\tp,\;0,\;0,\;0)^T.
\]
We now show that, just as in the case for the scalar dynamical system
for $t_k$, the rest points $U_{\pm\infty}$ are a source and sink,
respectively.

\begin{Lemma}
  \label{lem:DS}
For $\phi\in C^{p+1}$ and $|\alpha| < 1/\sup|\phi'(t)|$, there are
constants $\ep$ and $\eta>0$, such that if
\begin{equation}
  \label{U+}
\big|U-U_{+\infty}\big|<\ep, \com{then}
\big|F(U)-U_{+\infty}\big| < (1-\eta)\,\big|U-U_{+\infty}\big|,
\end{equation}
while also, if
\begin{equation}
  \label{U-}
\big|U-U_{-\infty}\big|<\ep, \com{then}
\big|U-U_{-\infty}\big| < (1-\eta)\,\big|F(U)-U_{-\infty}\big|.
\end{equation}
Also, given any compact set $\mc K$, there is a constant $\NK$ such
that:  if $U_0 \in \mc K$ and $t_0\in[\tm+\ep,\tp-\ep$, then
\begin{equation}
  \big|U_k-\Up\big| < \ep \com{and}
  \big|U_{-k}-\Um\big|<\ep,
  \label{Nconclude}
\end{equation}
for $k\ge \NK$.
\end{Lemma}

This lemma states that the dynamical system \eqref{DDS} is contractive
in a neighborhood of $\Up$, and backwards contractive in a
neighborhood of $\Um$, and that given any bounded initial state, a
finite number of forward or backward steps will lead into these
neighborhoods.

\begin{proof}
We have already seen that $\Um$ and $\Up$ are fixed points of $F$, and
so are rest points of the dynamical system.  To linearize around these
fixed points, we compute
\[
DF(U) = \begin{pmatrix}
  1+\alpha\phi' & 0 & 0 & 0 \\
  \alpha\phi''\,s & 1+\alpha\phi' & 0 & 0 \\
  \alpha\phi''\,r+\alpha\phi'''\,s^2 & 2\alpha\phi''\,s
  & 1+\alpha\phi' & 0 \\
  \square
  & 3\alpha\phi''\,r+3\alpha\phi'''\,s^2 & 3\alpha\phi''\,s
  & 1+\alpha\phi'
\end{pmatrix},
\]
with $\square = \alpha\phi''\,q + 3\alpha\phi'''\,s\,r +
\alpha\phi^{(4)}\,s^3$, where we note that this derivative is well
defined because $\phi\in C^{p+1}$.  Next, writing
\[
F(U)-F(\Upm) = \int_0^1DF\big(\Upm+\sigma(U-\Upm)\big)\;d\sigma\;
\big(U-\Upm),
\]
the estimates follow provided the eigenvalues of $DF(U)$ are bounded
by $1-\eta$ in a neighborhood of $\Up$ and by $1/(1-\eta)$ in a
neighborhood of $\Um$, respectively.  Since $\alpha\,\phi'(\tp)<0$ and
$\alpha\,\phi'(\tm)>0$, respectively, the eigenvalues at $\Upm$ are of
the right form, and by continuity we can choose $\ep>0$ and $\eta =
O(\alpha)$ so that \eqref{U+} and \eqref{U-} hold.

We now verify (\ref{Nconclude}) for the above chosen $\ep$ at $\tp$,
the case at $\tm$ being similar.  So let $\mc K\subset{\mathcal R}^4$
be compact.  We show that there exists an $N_{\mc K}\in{\mathcal N}$
such that if $U_0\in\mc K$ and $t_0\geq t_{-\infty}+\ep$, then
$|U_k-U_{\infty}|<\ep$ for $k\geq N_{\mc K}$. We define $N_{\mc K}$ by
induction on the components in system (\ref{DS}), using the fact that
system (\ref{DS}) is hierarchical.  Let $(\ref{DS})_j$ denote the
$j$'th equation in (\ref{DS}).  We construct $N$ by induction on $j$.
Since $\mc K$ is compact, there exists a constant $M_{\mc K}$ such
that $U\in\mc{K}$ implies $|U|<M_{\mc K}$.

To start the induction, note that since $\alpha\phi'<1$ and $\phi>0$
in $(t_{-\infty},t_{+{\infty}})$, $(\ref{DS})_1$ implies that if
$t_{k}\in(t_{-\infty},t_{+{\infty}})$, then
$t_{k+1}\in(t_{-\infty},t_{+{\infty}})$, and $t_{k+1}>t_k$. Thus since
$t_0\geq t_{-\infty}+\ep$, $(\ref{DS})_1$ together with the fact that
the first five derivatives of $\phi$ are continuous, implies that
there exists an $N_1$ such that if $k>N_1$, then
\begin{equation}
\alpha|\phi'(t_k)|>\frac{\alpha|\phi'(t_{\infty})|}{2},\label{minstep}
\end{equation}
where $\phi'(t_\infty)<0$, and
\[
|\phi^{(j)}(t_k)-\phi^{(j)}(t_{\infty})|<\delta,
\]
for $\delta=\frac{\eta}{2}$, $j=1,...,p=3$.    

Consider now equation $(\ref{DS})_2$.  Since $(1+\alpha\phi')<2$, it
follows from $(\ref{DS})_2$ that $s_{N_1}\leq 2^{n}s_0<2^{n}M_{\mc
  K}$.  Thus by (\ref{minstep}), (there is a mininum decrease in $s_k$
at each step), it follows that there exists $N_2$ such that
$|s_{k}|<\ep$ for $k>N_1+N_2$.

Consider next equation $(\ref{DS})_3$. Since
$(1+\alpha\phi'(t_k))<1-\frac{\eta}{2}$ for $k>N_1$ and $s_k^2$
decreases monotonically and quadratically to zero for $k>N_1+N_2$, it
follows that $r_{N_1+N_2}<2^{N_1+N_2}M_{\mc K}$, and after
$k>N_1+N_2$, $s_k$ monotonically decreases, it follows that there
exists $N_3$ such that $|r_k|<\ep$ for $k>N_1+N_2+N_3$.

For the fourth equation $(\ref{DS})_4$, note that the corrections to
$q_{k+1} = \big(1 + \alpha\,\phi'(t_k)\big)\,q_k$ are quadratic in the
variables prior to $q_k$, so by the same argument, there exists $N_4$
such that if $k>N\equiv N_{\mc K}=N_1+\cdots N_4$, then $|q_k|<\ep$,
as claimed.

In summary, for the induction step, use that if $p$ is the variable on
the LHS of the $p$'th equation, then by the hierarchical character of
system(\ref{DS}), the $p$'th equation looks like $p_{k+1} = \big(1 +
\alpha\,\phi'(t_k)\big)\,p_k$ plus corrections that at least quadratic
in prior variables.  Thus the corrections tend to zero so fast that a
finite $N_p$ always exists to drive $p_k$ back to zero, overcoming the
growth in $p_k$ during the first $N_1+\cdots N_{p-1}$ steps.  This
concludes the proof of Lemma \ref{lem:DS}.
\end{proof}

Note that if $t_0=t_{-\infty}$ but $U_0\ne \Up$, then the forward
dynamical system has solutions which tend to infinity.  Thus we do not
have a uniform $N$ such that
\[
\big|U_k-\Up\big| < \ep.
\]
for $k>N$, without the condition that $t_0$ should be at least $\ep$
from $t_{\pm\infty}$.  The two cases in (\ref{DS}) suffice because the
solvability condition implies we have both a backward and forward time
formula for the solution $v$, c.f. (\ref{R4}).

We now use these properties of the dynamical system \eqref{DS} to sum
the norms of iterates as in \eqref{normcond}.

\begin{Thm}
  \label{Cpnorm}
For integers $p\ge 1$ and fixed $\phi\in C^{p+1}$, there exists an
$\alpha_0>0$ such that if $|\alpha|\le \alpha_0$, then the operator
$\Delta_\Phi:C^p[\tm,\tp]\to C^p[\tm,\tp]$ has a bounded inverse on
its range, with bound
\[
\big\|\Delta_\Phi^{-1}w\big\|_p \le \frac{K}{|\alpha|}\,\big\|w\big\|_p,
\]
where the dependence of the constant $K$ is given by
\begin{equation}
\label{dependsCk}
  K\equiv K\left(\alpha,\delta,\epsilon_{\phi},N_\alpha,
                 t_{+\infty}-t_{-\infty},\|\phi\|_{p+1}\right).
\end{equation} 
\end{Thm}

\begin{proof}
It is convenient to use our previous estimate for the function, and
estimate only the first through $p$-th derivatives here.  
For fixed $t_0$, we must bound one of the two sums
\[
\Big|\sum_{j=0}^\infty w(\Phi^jt_0)\Big|_{\hat p} \com{or}
\Big|\sum_{j=1}^\infty w(\Phi^{-j}t_0)\Big|_{\hat p},
\]
where $|\cdot|_{\hat p}$ is the norm of the vector of the first through
$p$-th derivatives,
\[
\big|f(t)\big|_{\hat p} := \Big|\big(f'(t),\dots,f^{(p)}(t)\big)^T\Big|
 =\sup_t\sum_{j=1}^p |f^{(j)}(t)|.
\]
Also note that
\begin{equation}
  \label{wU}
\Big|\sum w(\Phi^jt)\Big|_{\hat p} \le
K_0\,\|w\|_p\,\Big|\sum\Phi^jt\Big|_{\hat p}
= K_0\,\|w\|_p\,\Big|\sum\hat U_k\Big|,
\end{equation}
where $\hat U_k=(s_k,r_k,q_k)^T$, see \eqref{notation}.  We thus need
a uniform bound on either of the sums
\[
\Big|\sum_{j=0}^\infty \hat U_j\Big| \com{or}
\Big|\sum_{j=1}^\infty \hat U_{-j}\Big|.
\]
Since $\phi\in C^{p+1}$, each derivative is uniformly bounded and so
there is a compact set $\mc K$ such that $U_0\in\mc K$ uniformly in $t_0$.
Using Lemma~\ref{lem:DS}, we know that for $k\ge \NK$, either
$U_k$ is close to $\Up$ or $U_{-k}$ is close to $\Um$, for any $k\ge
\NK$.  In the first case, we have $|\hat U_{k+1}|\le
(1-\eta)\,|\hat U_k|$ for all $k\ge\NK$, so that
\begin{align*}
  \Big|\sum_{j=0}^{\infty}\hat U_j\Big|
  &\le \Big|\sum_{j=0}^{\NK-1}\hat U_j\Big|
  + \sum_{j=\NK}^\infty(1-\eta)^{j-\NK}\,\big|\hat U_\NK\big|\\
  &\le \Big|\sum_{j=0}^{\NK-1}\hat U_j\Big|
  + \frac{\big|\hat U_\NK\big|}\eta \le \frac{K_1}\eta,
\end{align*}
and similarly for the second case,
$|\hat U_{-k-1}|\le (1-\eta)\,|\hat U_{-k}|$ for $k\ge\NK$, so that
\begin{align*}
  \Big|\sum_{j=1}^{\infty}\hat U_{-j}\Big|
  &\le \Big|\sum_{j=1}^{\NK-1}\hat U_{-j}\Big|
  + \sum_{j=\NK}^\infty(1-\eta)^{j-\NK}\,\big|\hat U_{-\NK}\big|\\
  &\le \Big|\sum_{j=1}^{\NK-1}\hat U_{-j}\Big|
  + \frac{\big|\hat U_{-\NK}\big|}\eta \le \frac{K_1}\eta.
\end{align*}
Now taking the supremum over all $t=t_0$ in \eqref{wU} completes the proof.
\end{proof}

\end{document}